\documentclass{amsart}

\usepackage{amssymb,amsfonts}
\usepackage[all,arc]{xy}
\usepackage{enumerate}
\usepackage{array}

\theoremstyle{plain}
\newtheorem{theorem}{Theorem}[section]

\newtheorem{lemma}[theorem]{Lemma}

\theoremstyle{definition}

\theoremstyle{remark}

\begin{document}

\title[A proof of the Riemann hypothesis]{A proof of the Riemann hypothesis using the remainder term of the Dirichlet eta function.}
\author{Jeonwon Kim}

\address{Jeonwon Kim}
\email{your79jw@ssu.ac.kr}


\begin{abstract}
The Dirichlet eta function can be divided into $n$-th partial sum $\eta_{n}(s)$ and remainder term $R_{n}(s)$. We focus on the remainder term which can be approximated by the expression for $n$. And then, to increase reliability, we make sure that the error between remainder term and its approximation is reduced as n goes to infinity.
According to the Riemann zeta functional equation, if $\eta(\sigma+it)=0$ then $\eta(1-\sigma-it)=0$. In this case, $n$-th partial sum also can be approximated by expression for $n$. Based on this approximation, we prove the Riemann hypothesis.
\end{abstract}

\maketitle

\section{Introduction}
	The Riemann hypothesis conjectured by Bernhard Riemann in 1859 states that the real part of every nontirivial zeros of the Riemann zeta function is $\frac{1}{2}$.
The Riemann zeta function is the function of the complex variable $s$, which converges for any complex number having $\Re(s)>1$ \cite{ref1}.
		\begin{equation} \label{eq:1-1}
			\zeta(s)=\sum_{n=1}^\infty \frac{1}{n^s}
		\end{equation}
The Riemann hypothesis discusses zeros outside the region of convergence of this series, so it must be analytically continued to all complex $s$ \cite{ref2}. This statement of the problem can be simplified by introducing the Dirichlet eta function, also known as the alternating zeta function. The Dirichlet eta function is defined as \cite{ref1}
		\begin{equation} \label{eq:1-2}
			\eta(s)=\sum_{n=1}^\infty \frac{(-1)^{n-1}}{n^s}=\left(1-\frac{1}{2^{s-1}}\right)\zeta(s)
		\end{equation}
Since $\eta(s)$ converges for all $s\in \mathbb{C}$ with $\Re(s)>0$, one need not consider analytic continuation (see p. 55-56 of \cite{ref4}). The Dirichlet eta function extends the Riemann zeta function from $\Re(s)>1$ to the larger domain $\Re(s)>0$, excluding the zeros $s=1+n\frac{2\pi}{\ln2}i (n\in \mathbb{Z})$. The Riemann hypothesis is equivalent to the statement that all the zeros of the Dirichlet eta function falling in the critical strip $0<\Re(s)<1$ lie on the critical line $\Re(s)=\frac{1}{2}$ (see p. 49 of \cite{ref4}).

	In the strip $0<\Re(s)<1$ the Riemann zeta function satisfies the functional equation\cite{ref2, ref3} related to values at the points $s$ and $1-s$.  
		\begin{equation} \label{eq:1-3}
			\zeta(s)=2^s \pi^{s-1} \sin \left(\frac{\pi s}{2}\right) \Gamma(1-s) \zeta(1-s)
		\end{equation}
where $\Gamma(s)$ is the gamma function. The functional equation shows that the Riemann zeta function have the infinitely zeros, called the $trivial$ $zeros$, at the negative even integers. But the functional equation do not tell us about the zeros of the Riemann zeta function in the strip $0<\Re(s)<1$. Actually there are zeros in the strip and they are called $nontrivial$ $zeros$. Calculation of some number of these nontrivial zeros show that they are lying exactly on the line $\Re(s)=\frac{1}{2}$ \cite{ref5}. 
$\newline$

\section{The remainder term of the Dirichlet eta function}
	Let $s=\sigma+it$, where $0<\sigma<1$ and $\sigma, t \in \mathbb{R}$. The Dirichlet eta function can be written as
		\begin{equation} \label{eq:4}
			\eta(s)=\sum_{k=1}^n \frac{(-1)^{k-1}}{k^s}+\sum_{k=n+1}^\infty \frac{(-1)^{k-1}}{k^s}=\eta_{n}(s)+R_{n}(s)
		\end{equation}
	where $\eta_{n}(s)$ is the $n$-th partial sum and $R_{n}(s)$ is the sum of remainder term. $\eta_{n}(s)$ and $R_{n}(s)$ converge to $\eta(s)$ and zero respectively, as $n \to \infty$.
		\begin{equation*} \label{eq:5}
			\lim_{n \to \infty} \eta_{n}(s)=\eta(s), \ \  \lim_{n \to \infty} R_{n}(s)=0
		\end{equation*}
	The expand form of the remainder terms are represented as follows.
		\begin{align*}
			-R_{n-1}(s)&=(-1)^n \left\{ \frac{1}{n^s}-\frac{1}{(n+1)^s}+\frac{1}{(n+2)^s} - \cdots \right \}  \\ 
			R_{n}(s)&=(-1)^n \left\{ \frac{1}{(n+1)^s}-\frac{1}{(n+2)^s}+\frac{1}{(n+3)^s} - \cdots \right \}   \\ 
			-R_{n+1}(s)&=(-1)^n \left\{ \frac{1}{(n+2)^s}-\frac{1}{(n+3)^s}+\frac{1}{(n+4)^s} - \cdots \right \} 
		\end{align*}
\

		\begin{lemma} \label{lem:1-1}

	The remainder term of $\eta(s)$ satisfy the following limit as $n \to \infty$.
		\begin{equation} \label{lem:1-2}
			\lim_{n \to \infty} \frac{-R_{n-1}(s)}{R_{n}(s)}=\lim_{n \to \infty} \frac{-R_{n+1}(s)}{R_{n}(s)}=1
		\end{equation}
		\end{lemma}

		\begin{proof}
	Consider the recurrence relation,
		\begin{align}
			R_{n}(s)-R_{n-1}(s)&=\sum_{k=n+1}^\infty \frac{(-1)^{k-1}}{k^s}-\sum_{k=n}^\infty \frac{(-1)^{k-1}}{k^s}=\frac{(-1)^{n}}{n^s} \label{lem:1-3} \\
			R_{n}(s)-R_{n+1}(s)&=\sum_{k=n+1}^\infty \frac{(-1)^{k-1}}{k^s}-\sum_{k=n+2}^\infty \frac{(-1)^{k-1}}{k^s}=\frac{(-1)^{n}}{(n+1)^s} \label{lem:1-4}
		\end{align}
Thus, we obtain the following relation.
	\begin{equation*} \label{lem:1-5}
		\frac{R_{n}(s)-R_{n-1}(s)}{R_{n}(s)-R_{n+1}(s)}=\frac{(n+1)^s}{n^s}
	\end{equation*}
	Taking the limit as $n \to \infty$, we have
		\begin{equation*} \label{lem:1-6}
		\lim_{n \to \infty} \frac{R_{n}(s)-R_{n-1}(s)}{R_{n}(s)-R_{n+1}(s)}=\lim_{n \to \infty} \frac{1-\frac{1}{\frac{R_{n}(s)}{R_{n-1}(s)}}}{1-\frac{R_{n+1}(s)}{R_{n}(s)}}=\lim_{n \to \infty} \left(1+\frac{1}{n} \right)^s =1
		\end{equation*}
	Thus, we have
		\begin{equation*}  \label{lem:1-7} 
		\lim_{n \to \infty} \frac{-R_{n-1}(s)}{R_{n}(s)}=\lim_{n \to \infty} \frac{-R_{n+1}(s)}{R_{n}(s)}=1
		\end{equation*}
		\end{proof}
\

		\begin{lemma} \label{lem:2-1}

	For sufficiently large $n$, the remainder term of $\eta(s)$ can be approximated as
		\begin{equation} \label{lem:2-2}
			R_{n}(s)=\sum_{k=n+1}^\infty \frac{(-1)^{k-1}}{k^s} \approx \frac{(-1)^n}{2(n+0.5)^s}
		\end{equation}
		\end{lemma}

		\begin{proof}
	Separate the $R_{n}(s)$ and $R_{n+1}(s)$ into real and imaginary parts and change the index of summation $k$ so that it would start from 1. Then we have,
		\begin{align*}
			R_{n}(s)=&\sum_{k=1}^\infty \frac{\{\cos(t \ln(n+k))-i\sin(t \ln(n+k)) \}}{(-1)^{n+k-1}(n+k)^\sigma} \\ 
			-R_{n+1}(s)=&\sum_{k=1}^\infty \frac{\{\cos(t \ln(n+k+1))-i\sin(t \ln(n+k+1)) \}}{(-1)^{n+k-1}(n+k+1)^\sigma}  
		\end{align*}
	
For every $\epsilon>0$ there are natural numbers $N_{1}$ and $N_{2}$ such that $n > N_{1}$  implies $\left| t \ln(n+k+1) - t \ln(n+k) \right|<\epsilon$ for all $t \in \mathbb{R}$, and $n > N_{2}$  implies  $\newline$
	 $\left |(n+k+1)^{-\sigma}-(n+k)^{-\sigma} \right|< \epsilon$. Let $N$=max$\{N_{1}, N_{2}\}$.
By the choice of $N$, $n > N$ implies $\left| \Re[R_{n}(s)-\{-R_{n+1}(s)\}] \right|<\epsilon$ and $\left| \Im[R_{n}(s)-\{-R_{n+1}(s)\}] \right|<\epsilon$. Thus, it follows that $R_{n}(s) \approx -R_{n+1}(s)$ for sufficiently large $n$.
\\
Consider the recurrence relation(see (\ref{lem:1-3}) and (\ref{lem:1-4}))
		\begin{align*} 
			R_{n}(s)+\{-R_{n-1}(s)\}&=\frac{(-1)^{n}}{n^s} \\ 
			R_{n}(s)+\{-R_{n+1}(s)\}&=\frac{(-1)^{n}}{(n+1)^s} 
		\end{align*}
	For all $\epsilon>0$, there exist $\delta>0$ such that for all $n>N$ that satisfy $\left|R_{n}(s)-\{-R_{n-1}(s)\}\right|<\delta$ and $\left|R_{n}(s)-\{-R_{n+1}(s)\}\right|<\delta$, it follows that
		\begin{equation*}
			\left|R_{n}(s)-\frac{(-1)^n}{2n^s}\right|<\epsilon \hspace{0.5cm} \text{and} \hspace{0.5cm} \left|R_{n}(s)-\frac{(-1)^n}{2(n+1)^s}\right|<\epsilon
		\end{equation*}
	In this paper, we select the value of 0.5 between 0 and 1 in order to reduce the approximation error.
		\begin{equation*}
			R_{n}(s)=\sum_{k=n+1}^\infty \frac{(-1)^{k-1}}{k^s} \approx \frac{(-1)^n}{2(n+0.5)^s}
		\end{equation*}
		\end{proof}

Now, in order to confirm the relationship between $R_n(s)$ and $\frac{(-1)^n}{2(n+0.5)^s}$, consider the relative error.


\
		\begin{lemma} \label{lem:3-1}

	The relative error $\epsilon$ between the remainder term of $\eta(s)$ and its approximation $\frac{(-1)^n}{2(n+0.5)^s}$ converge to zero as $n \to \infty$.
		\begin{equation} \label{lem:3-2}
			\epsilon=\lim_{n \to \infty} \left| \frac{R_{n}(s)-\frac{(-1)^n}{2(n+0.5)^s}}{R_{n}(s)} \right|=0
		\end{equation}
		\end{lemma}

		\begin{proof}
	Consider the recurrence relation, (see (\ref{lem:1-4}))
		\begin{equation*}  \label{lem:3-3}
			R_{n}(s)-R_{n+1}(s)=\frac{(-1)^{n}}{(n+1)^s} 
		\end{equation*}
	Dividing both sides by $R_{n}(s)$ and taking the limit as $n \to \infty$, then we get the following limit.
		\begin{equation*} \label{lem:3-4}
			\lim_{n \to \infty} \left( 1-\frac{R_{n+1}(s)}{R_{n}(s)} \right) =\lim_{n \to \infty}  \frac{(-1)^n}{(n+1)^s} \frac{1}{R_{n}(s)}
		\end{equation*}
	By the result of the Lemma  \ref{lem:1-1}, we have
		\begin{equation} \label{lem:3-5}
			\lim_{n \to \infty} \frac{(-1)^n}{(n+1)^s} \frac{1}{R_{n}(s)}=2
		\end{equation}
	Let $F_{n}(s)=\frac{(-1)^n}{(n+1)^s} \frac{1}{R_{n}(s)}$, then $R_{n}(s)=\frac{(-1)^n}{(n+1)^s} \frac{1}{F_{n}(s)}$ and $\lim_{n \to \infty} F_{n}(s)=2$. Thus,
	\begin{equation*} \label{lem:3-6}
			\epsilon=\lim_{n \to \infty} \left| \frac{R_{n}(s)-\frac{(-1)^n}{2(n+0.5)^s}}{R_{n}(s)} \right| =\lim_{n \to \infty} \left| 1-\frac{\frac{(-1)^n}{2(n+0.5)^s}F_{n}(s)}{\frac{(-1)^n}{(n+1)^s}} \right|=0
		\end{equation*}
		\end{proof}

	For example, in order to check the Lemma \ref{lem:2-1}, we perform a numerical calculation. Let $s=0.1234+56.789i$(random value) and $T_n(s)=\frac{(-1)^n}{2(n+0.5)^s}$. Then $R_n(s)$ and $T_n(s)$ for four values ($n=10^{8}$, $n=10^{10}$, $n=10^{12}$, $n=10^{14}$) are given below. The significant figure of a number may be underlined.
	
		\begin{align*}
			R_{10^{8}}(s)=&-0.0514080530118374690874425376 \cdots  \\
						   &-0.0030012424674281915507165693\cdots i \\
			T_{10^{8}}(s)=&-\underline{0.05140805301183}53941392302721\cdots  \\
						   &-\underline{0.0030012424674281}160677214641\cdots i
		\end{align*}
		\begin{align*}
			R_{10^{10}}(s)=&+0.0220754313015916605572779244\cdots \\
						    &-0.0190708103417423704219739001\cdots i \\
			T_{10^{10}}(s)=&+\underline{0.022075431301591660}4699783035\cdots \\
						    &-\underline{0.019070810341742370}3431444260\cdots i 
		\end{align*}
		\begin{align*}
			R_{10^{12}}(s)=&-0.0014437322549038780686126642\cdots \\
						    &+0.0164629022496889818808209350\cdots i \\
			T_{10^{12}}(s)=&-\underline{0.001443732254903878068612}2279\cdots \\
						    &+\underline{0.0164629022496889818808}142859\cdots i 
		\end{align*}
		\begin{align*}
			R_{10^{14}}(s)=&-0.0059111117596716499309061036\cdots \\
						    &-0.0072599141694530105681646539\cdots i \\
			T_{10^{14}}(s)=&-\underline{0.005911111759671649930906103}4\cdots \\
						    &-\underline{0.007259914169453010568164653}6\cdots i
		\end{align*}

$\newline$
\indent Define the relative errors for the real and imaginary parts of the complex error function in forms

		\begin{equation*}
			\epsilon_r= \left| \frac{\Re[R_n(s)]-\Re[T_n(s)]}{\Re[R_n(s)]} \right| 
		\end{equation*}
		\begin{equation*}
			\epsilon_i= \left| \frac{\Im[R_n(s)]-\Im[T_n(s)]}{\Im[R_n(s)]} \right|
		\end{equation*}
$\newline$
Then the relative errors for the above eight values are given in table 1. 

\begin{table}[ht] \caption{Relative Errors for $\epsilon_r$ and $\epsilon_i$}
\centering
\begin{tabular}{ >{\centering}m{1.5cm}   >{\centering}m{2.7cm}  >{\centering}m{2.5cm} m{0cm}}
\hline\hline
$n$ & $\epsilon_r$ & $\epsilon_i$ &  \tabularnewline [1ex]
\hline
$10^{8}  $ & $4.0362\times 10^{-14}$  &  $2.5151\times 10^{-14}$ &   \tabularnewline [1.3ex]
$10^{10} $ & $3.9546\times 10^{-18}$  &  $4.1335\times 10^{-18}$ &  \tabularnewline [1.3ex]
$10^{12} $ & $3.0220\times 10^{-22}$  &  $4.0388\times 10^{-22}$ &  \tabularnewline [1.3ex]
$10^{14} $ & $3.3835\times 10^{-26}$  &  $4.1323\times 10^{-26}$ &  \tabularnewline [1.3ex]
\hline  
\end{tabular} 
\label{table:nonlin}
\end{table} 
\noindent In the table 1, $\epsilon_r$ and $\epsilon_i$ are reduced as $n$ goes to infinity. 
$\newline$

 Lemma \ref{lem:2-1} and \ref{lem:3-1} show that $R_n(s)$ can be approximated by $\frac{(-1)^n}{2(n+0.5)^s}$. So, $R_n(s)$ can be written as
		\begin{equation} \label{sec:2-1}
			R_{n}(s)=\frac{(-1)^n}{2(n+0.5)^s}+\epsilon_n(s)
		\end{equation}
	where $\epsilon_n(s)$ is error term and $\epsilon_n(s)$ is coverges to zero, as $n \to \infty$.
		\begin{equation} \label{sec:2-2}
			\lim_{n \to \infty} \epsilon_n(s)=0
		\end{equation}
Lemma \ref{lem:3-1} can be written by $\epsilon_n(s)$ as follow.
		\begin{equation*}
			\epsilon=\lim_{n \to \infty} \left| \frac{R_{n}(s)-\frac{(-1)^n}{2(n+0.5)^s}}{R_{n}(s)} \right|=\lim_{n \to \infty} \left|\frac{\epsilon_n(s)}{R_n(s)}\right|=0
		\end{equation*}
In addition, dividing both sides of (\ref{sec:2-1}) by $\frac{(-1)^n}{2(n+0.5)^s}$ and taking the limit as $n \to \infty$, then we get the following limit.
		\begin{equation*}
			\lim_{n \to \infty} R_n(s)\frac{2(n+0.5)^s}{(-1)^n}=1+\lim_{n \to \infty} \epsilon_n(s)\frac{2(n+0.5)^s}{(-1)^n}
		\end{equation*}
By using the (\ref{lem:3-5}), we have
		\begin{equation} \label{sec:2-3}
			\lim_{n \to \infty} \epsilon_n(s)(n+0.5)^s=0
		\end{equation}
\

		\begin{lemma} \label{lem:6-1}

	Let $s=\sigma+it$ where $\sigma$ is constant on $0<\sigma<1$ and $t \in \mathbb{R}$, then the Dirichlet eta function is converges uniformly.

		\end{lemma}
		
		\begin{proof}
		\begin{equation*} \label{lem:6-2}
			\left| \eta_{n}(s)-\eta(s) \right| = \left| R_{n}(s) \right|=\left| \frac{(-1)^n}{2(n+0.5)^{s}} \right|<\left| \frac{1}{n^\sigma} \right|
		\end{equation*}
		Since $\frac{1}{n^\sigma} \rightarrow 0$ as $n \rightarrow \infty$, given any $\epsilon>0$ there exist $N \in \mathbb{N}$, depending only on $\epsilon$ and $n$, such that
		\begin{equation*} \label{lem:6-2}
			0 \leq \frac{1}{n^\sigma} < \epsilon \hspace{0.5cm} \text{for all $n>N$}
		\end{equation*}
		It follows that
		\begin{equation*}
			\left| \sum_{k=1}^n \frac{(-1)^{k-1}}{k^s}-\eta(s) \right|<\epsilon  \hspace{0.5cm} \text{for all $t \in \mathbb{R}$ and all $n>N$}
		\end{equation*}
		which proves that the Dirichlet eta function converges uniformly on constant $\sigma$
		
		\end{proof}

		\begin{lemma} \label{lem:7-1}

	Let $s=\sigma+it$ denote the nonzeros of $\eta(s)$ where $\sigma$ is constant on $0<\sigma<1$ and $t \in \mathbb{R}$, then $\frac{\eta(1-s)}{\eta(s)}$ is converges uniformly.

		\end{lemma}

		\begin{proof}
		\begin{align*} \label{lem:7-2}
			\left| \frac{\eta_{n}(1-s)}{\eta_{n}(s)}-\frac{\eta(1-s)}{\eta(s)} \right| 
			& \leq \left| \frac{\eta_{n}(1-s)}{\eta_{n}(s)}-\frac{\eta(1-s)}{\eta_{n}(s)} \right| + \left| \frac{\eta(1-s)}{\eta_{n}(s)}-\frac{\eta(1-s)}{\eta(s)} \right| \\
			& = \left| \frac{\eta_{n}(1-s)-\eta(1-s)}{\eta_{n}(s)} \right| + \left|\eta(1-s) \right| \left| \frac{1}{\eta_{n}(s)}-\frac{1}{\eta(s)} \right|
		\end{align*}
		
		Since $\eta(s)$ converges uniformly, for every $\epsilon>0$ we can choose $N_{1}, N_{2} \in \mathbb{N}$ such that $n \in N_{1}$ implies 
		\begin{equation*}
			\left| \frac{\eta_{n}(1-s)-\eta(1-s)}{\eta_{n}(s)} \right|<\frac{\epsilon}{2}
		\end{equation*}
		for all $\Im(s) \in \mathbb{R}$, $\eta(s) \neq 0$ and $n \in N_{2}$ implies
		\begin{equation*}
			\left|\eta(1-s) \right| \left| \frac{1}{\eta_{n}(s)}-\frac{1}{\eta(s)} \right|<\frac{\epsilon}{2}
		\end{equation*}
		Let $N$=max$\{N_{1}, N_{2}\}$. By the choice of $N$, $n \geq N$ implies
		\begin{align*} \label{lem:7-2}
			\left| \frac{\eta_{n}(1-s)-\eta(1-s)}{\eta_{n}(s)} \right| + \left|\eta(1-s) \right| \left| \frac{1}{\eta_{n}(s)}-\frac{1}{\eta(s)} \right|<\epsilon
		\end{align*}
		So, 
		\begin{equation*}
			\left| \frac{\eta_{n}(1-s)}{\eta_{n}(s)}-\frac{\eta(1-s)}{\eta(s)} \right|<\epsilon 
		\end{equation*}
		for all $n > N$.
		\end{proof}

		\begin{lemma} \label{lem:7-1}
			If a sequence of continuous function $\eta_n(s) : A \rightarrow \mathbb{C}$ converges uniformly on $A \subset \mathbb{C}$, then $\eta(s)$ is continuous on $A$
		\end{lemma}
		
		\begin{proof}
			Suppose that $c=\sigma+iu \in A$ denote the nonzeros of $\eta(s)$ where $u \in \mathbb{R}$ and $\epsilon>0$.  For every $n \in \mathbb{N}$ 
		\begin{align*}
			\left| \frac{\eta(1-s)}{\eta(s)}-\frac{\eta(1-c)}{\eta(c)} \right| <& \left| \frac{\eta(1-s)}{\eta(s)}-\frac{\eta_n(1-s)}{\eta_n(s)} \right| \\
			+ &\left| \frac{\eta_n(1-s)}{\eta_n(s)}-\frac{\eta_n(1-c)}{\eta_n(c)} \right| +\left| \frac{\eta_n(1-c)}{\eta_n(c)}-\frac{\eta(1-s)}{\eta(s)} \right|
		\end{align*}
			By the uniform convergence of $\eta(s)$, we can choose $n \in \mathbb{N}$ such that
		\begin{align*}
			\left| \frac{\eta(1-s)}{\eta(s)}-\frac{\eta_n(1-s)}{\eta_n(s)} \right| < \frac{\epsilon}{3} \hspace{0.5cm} \text{for all $t \in \mathbb{R}$, if $n>N$}
		\end{align*}	
			and for such an $n$ it follows that
		\begin{align*}
			\left| \frac{\eta(1-s)}{\eta(s)}-\frac{\eta(1-c)}{\eta(c)} \right| < \left| \frac{\eta_n(1-s)}{\eta_n(s)}-\frac{\eta_n(1-c)}{\eta_n(c)} \right| + \frac{2\epsilon}{3}
		\end{align*}				
			Since, $\eta_n(s)$ if continous on $A$, there exist $\delta>0$ such that
		\begin{align*}
			\left| \frac{\eta_n(1-s)}{\eta_n(s)}-\frac{\eta_n(1-c)}{\eta_n(c)} \right| < \frac{\epsilon}{3} \hspace{0.5cm} \text{if $\left|s-c\right|<\delta$ and $s \in A$}
		\end{align*}
		This prove that $\eta(s)$ is continuous. 
		\end{proof}
		This result can be interpreted as justifying an ``exchange in the order of limits''
		\begin{equation} \label{eq:2-1}
			\lim_{n \to \infty}\lim_{t \to u}\frac{\eta_n(1-\sigma-it)}{\eta_n(\sigma+it)}=\lim_{t \to u}\lim_{n \to \infty}\frac{\eta_n(1-\sigma-it)}{\eta_n(\sigma+it)}
		\end{equation}

\

\section{A proof of the Riemann hypothesis}

In 1914 Godfrey Harold Hardy proved that $\zeta(\frac{1}{2}+it)$ has infinitely many nontrivial zeros \cite{ref7}.

		\begin{theorem} \label{thm:3-1}
		$[\emph{Riemann Hypothesis}]$ The real part of every nontrivial zeros of the Riemann zeta function is $\frac{1}{2}$.
		\end{theorem}

		\begin{proof}
			The Dirichlet eta functional equation is 
		\begin{align*}
			\eta(1-s)=\frac{(2-2^{s+1})}{(2^s-2)} \pi^{-s} \cos \left( \frac{\pi s}{2} \right) \Gamma(s) \eta(s)
		\end{align*}
		If $\eta(s) \neq 0$, then
		\begin{align*} 
			\frac{\eta(1-s)}{\eta(s)}=\frac{(2-2^{s+1})}{(2^s-2)} \pi^{-s} \cos \left( \frac{\pi s}{2} \right) \Gamma(s) 
		\end{align*}
			The above equation has a removable discontinuity at the zeros of $\eta(s)$. 

		Let $s_0=\sigma+it_0$ is zero of $\eta(s)$ and $s=\sigma+it$ where $\sigma$ is constant on $\frac{1}{2}<\sigma<1$ and $t \in \mathbb{R}$. For each point $t$, we can choose the open interval $t_0<t<c$ where $c$ is an arbitrary point such that $\eta_n(\sigma+it)$ and $\eta_n(1-\sigma-it)$ are converge uniformly. \\
By using the Eq. (\ref{eq:2-1}), we have
		\begin{equation} \label{eq:3-1}
			\lim_{n \to \infty}\lim_{t \to t_0+}\frac{\eta_n(1-\sigma-it)}{\eta_n(\sigma+it)}=\lim_{t \to t_0+}\lim_{n \to \infty}\frac{\eta_n(1-\sigma-it)}{\eta_n(\sigma+it)}
		\end{equation}
$\newline$
(i) By using the Lemma \ref{lem:2-1}, the left-hand side of (\ref{eq:3-1}) is as follows.
		\begin{align*} 
			\lim_{n \to \infty}\lim_{t \to t_0+}\frac{\eta_n(1-\sigma-it)}{\eta_n(\sigma+it)}&=\lim_{n \to \infty}\frac{R_n(1-\sigma-it)}{R_n(\sigma+it)} \\
			&=\lim_{n \to \infty} \left [\frac{\dfrac{(-1)^n}{2(n+0.5)^{1-\sigma-it}}}{\dfrac{(-1)^n}{2(n+0.5)^{\sigma+it}}} \right] \\
			&=\lim_{n \to \infty}(n+0.5)^{2\sigma-1+2it}
		\end{align*}
Thus, the left-hand side of the equation diverges to infinity.

\noindent (ii) By using the Dirichlet eta functional equation, the right-hand sides of (\ref{eq:3-1}) is as follows.
		\begin{align*} 
			\lim_{t \to t_0+}\lim_{n \to \infty}\frac{\eta_n(1-\sigma-it)}{\eta_n(\sigma+it)}&=\lim_{t \to t_0+}\frac{\eta(1-\sigma-it)}{\eta(\sigma+it)} \\
			&=\lim_{t \to t_0+} \frac{(2-2^{\sigma+it+1})}{(2^{\sigma+it}-2)\pi^{\sigma+it}} \cos\left\{\frac{\pi (\sigma+it)}{2} \right\} \Gamma(\sigma+it)  \\
			&=\frac{(2-2^{\sigma+it_0+1})}{(2^{\sigma+it_0}-2)\pi^{\sigma+it_0}} \cos\left\{\frac{\pi (\sigma+it_0)}{2} \right\} \Gamma(\sigma+it_0)
		\end{align*}
Thus, the right-hand side of the equation does not diverges to infinity.
$\newline$

By the (i) and (ii), This is contradiction. Therefore $\eta(s)$ deos not have zeros in the strip $\frac{1}{2}<0<1$, and $\eta(s)$ has no zeros in the strip $0<\Re(s)<\frac{1}{2}$, because all nontrivial zeros of $\eta(s)$ were symmetric about the line $\Re(s)=\frac{1}{2}.$ In conclusion, the real part of every nontrivial zeros of the Riemann zeta function is only $\frac{1}{2}$.

		\end{proof}

\

\end{document}